\documentclass[11pt]{amsart}
\usepackage[left=1.25in,right=1.25in]{geometry}
\usepackage{amsthm}
\usepackage{amsmath}
\usepackage{mathtools}
\usepackage{amssymb}
\usepackage{hyperref}
\usepackage{tikz}
\usepackage{tikz-cd}
\usetikzlibrary{decorations.markings}
\usepackage{enumitem}
\usepackage{amsfonts}
\usepackage{hyperref}
\usepackage{todonotes}
\usepackage[utf8]{inputenc}

\newcommand\precdot{\mathrel{\ooalign{$\prec$\cr
  \hidewidth\raise0.001ex\hbox{$\cdot\mkern0.6mu$}\cr}}}
\newcommand{\R}{\mathbb{R}}

\newtheorem{theorem}{Theorem}[section]
\newtheorem{prop}[theorem]{Proposition}
\newtheorem{conj}[theorem]{Conjecture}
\newtheorem{lemma}[theorem]{Lemma}
\newtheorem{claim}[theorem]{Claim}
\newtheorem{cor}[theorem]{Corollary}

\theoremstyle{definition}
\newtheorem{ex}[theorem]{Example}
\newtheorem{quest}[theorem]{Question}

\newtheorem{defin}[theorem]{Definition}

\theoremstyle{remark}
\newtheorem{remark}{Remark}

\DeclareMathOperator{\Red}{Red}
\DeclareMathOperator{\diam}{diam}
\DeclareMathOperator{\ro}{ro}
\DeclareMathOperator{\sort}{sort}
\DeclareMathOperator{\rro}{rro}
\DeclareMathOperator{\id}{id}

\newcommand{\mc}{\mathcal}

\renewcommand{\R}{\mathbb{R}}

\title{Diameters of graphs of reduced words and rank-two root subsystems}

\author{Christian Gaetz}
\thanks{C.G. was supported by a National Science Foundation Graduate Research Fellowship under Grant No. 1122374.}
\address{Department of Mathematics, Massachusetts Institute of Technology, Cambridge, MA 02139}
\email{\href{mailto:crgaetz@gmail.com}{{\tt crgaetz@gmail.com}}}
\author{Yibo Gao}
\email{\href{mailto:gaoyibo@mit.edu}{{\tt gaoyibo@mit.edu}}}

\date{\today}

\begin{document}
\begin{abstract}
We study the diameter of the graph $G(w)$ of reduced words of an element $w$ in a Coxeter group $W$ whose edges correspond to applications of the Coxeter relations. We resolve conjectures of Reiner--Roichman \cite{reiner-roichman} and Dahlberg--Kim \cite{dahlberg-kim} by proving a tight lower bound on this diameter when $W=S_n$ is the symmetric group and by characterizing the equality cases. We also give partial results in other classical types which illustrate the limits of current techniques.
\end{abstract}

\maketitle

\section{Introduction}
Given an element $w$ in a Coxeter group $W$, its set $\Red(w)$ of reduced words is very well-studied and important in a variety of algebraic, combinatorial, and geometric contexts. There is a natural graph structure $G(w)$ on $\Red(w)$, with two reduced words connected by an edge whenever they differ by a single application of one of the defining Coxeter relations
\[
\underbrace{s_is_j \cdots}_{m_{ij}} = \underbrace{s_js_i \cdots}_{m_{ij}}.
\]

A foundational result of Tits \cite{tits-words} says that the graph $G(w)$ is always connected. Further important work by Stanley \cite{stanley-reduced-words}, Reiner \cite{reiner-degree}, Tits \cite{tits-local}, and many others studied the cardinality, average degree, and topology of these graphs.

For finite Coxeter groups $W$, the diameter of $G(w)$ was first studied asymptotically by Dehornoy--Autord \cite{dehornoy-autord} and then exactly by Reiner--Roichman \cite{reiner-roichman} and Dahlberg--Kim \cite{dahlberg-kim}. Reiner and Roichman's insight was to study distances and diameters in $G(w)$ in relation to a certain set $L_2(w)$ of codimension-two subspaces in the associated reflection hyperplane arrangement. Intuitively, these subspaces are potential geometric obstructions to transforming one reduced word (viewed as a geodesic between chambers of the arrangement) into another. Our first main result resolves a conjecture of Reiner--Roichman by establishing a tight lower bound on the diameter of $G(w)$ in terms of $|L_2(w)|$ in the case of the symmetric group.

\begin{theorem}[Conjectured by Reiner--Roichman \cite{reiner-roichman}]
\label{thm:type-A-lower-bound}
For any permutation $w$ in the symmetric group $S_n$ we have:
\begin{equation}
\label{eq:main-thm-statement}
\frac{1}{2}\left| L_2(w) \right| \leq \diam(G(w)).
\end{equation}
\end{theorem}

\begin{remark}
Reiner and Roichman also conjecture a similar lower bound on $\diam(G(w))$ in type $B_n$ and upper bounds on $\diam(G(w))$ in both types. These conjectures remain open (see Sections~\ref{sec:other-types} and \ref{sec:upper-bounds-D} for discussion and partial results).
\end{remark}

Given permutations $w=w_1\ldots w_n$ and $u=u_1 \ldots u_m$ written in one-line notation, the \emph{direct sum} $w \oplus u$ is the permutation in $S_{n+m}$ with one-line notation $w_1\ldots w_n (n+u_1) \ldots (n+u_m)$ and the \emph{skew sum} $w \ominus u$ is the permutation $(n+w_1)\ldots (n+w_n) u_1 \ldots u_m$. We write $\id_n$ for the identity permutation in $S_n$.

Our second main result characterizes the equality case in (\ref{eq:main-thm-statement}), resolving a conjecture of Dahlberg--Kim.

\begin{theorem}[Conjectured by Dahlberg--Kim \cite{dahlberg-kim}]
\label{thm:type-A-equality-cases}
Equality is achieved in (\ref{eq:main-thm-statement}) if and only if $w$ is of the form $\id_i \oplus (\id_j \ominus \id_k) \oplus \id_{\ell}$ for some $i,j,k,\ell$.
\end{theorem}

Section~\ref{sec:background} gives needed definitions and background on hyperplane arrangements, root systems and Weyl groups, and permutation patterns. Section~\ref{sec:symmetric-group} applies a new argument in terms of permutation patterns in order to prove Theorems~\ref{thm:type-A-lower-bound} and \ref{thm:type-A-equality-cases}; in Section~\ref{sub:3412-conjecture} we conjecture a new \emph{upper} bound on $\diam(G(w))$ expressed in terms of permutation patterns which strengthens conjectures from \cite{dahlberg-kim, reiner-roichman}. Section~\ref{sec:other-types} discusses the applicability of our methods to proving lower bounds in the other classical types. In type $B_n$ we use these methods to prove a new lower bound, but one which is weaker than the conjectured value; in type $D_n$ we give an example to show that no analogous lower bound exists. Finally, in Section~\ref{sec:upper-bounds-D} we study $G(w_0)$ in type $D_n$. We show that Reiner and Roichman's technique for calculating the diameter of $G(w_0)$ cannot work in this case, but we prove an upper bound on $\diam(G(w_0))$ which agrees with the predicted value up to leading order.

\section{Background}
\label{sec:background}
\subsection{Hyperplane arrangements}\label{sub:background-hyperplane}
We mainly follow the conventions in \cite{reiner-roichman}. Readers are also referred to \cite{Stanley-hyperplane} for a detailed exposition on hyperplane arrangements.

Let $\mathcal{A}=\{H_1,\ldots,H_N\}$ be a hyperplane arrangement in $\R^n$ that is \textit{central} and \textit{essential} (that is, $\bigcap_{i=1}^N H_i=\{0\}$). For $d=0,1,\ldots,n$, let $L_d$ be the set of codimension-$d$ subspaces that are intersections of the hyperplanes in $\mathcal{A}$ and let $\mathcal{C}$ be the set of chambers of $\mathcal{A}$. Two chambers are \textit{adjacent} if they are separated by exactly one hyperplane, and this adjacency gives rise to a graph $G_1$ with vertices $\mathcal{C}$ and edges between adjacent chambers.

A \textit{gallery} or \textit{geodesic} between chambers $c$ and $c'$ is a shortest path in the graph $G_1$. Let $\mathcal{R}(c,c')$ denote the set of galleries from $c$ to $c'$. For any intersection subspace $X\in L_d$, the \textit{localized arrangement} is a hyperplane arrangement in the quotient space $\R^d/X$, defined as \[\mathcal{A}_X\coloneqq\{H/X\:|\: H\in\mathcal{A}, H\supset X\}.\] Any gallery $r\in\mathcal{R}(c,c')$ descends naturally to a gallery $r/X\in\mathcal{R}(c/X,c'/X)$. For two galleries $r,r'\in\mathcal{R}(c,c')$, we say that a codimension-two subspace $X\in L_{2}$ \textit{separates} $r$ and $r'$, if $r/X\neq r'/X$. Given either a pair $c,c'$ of chambers or a pair $r,r'$ of galleries between the same pair of chambers, the associated \emph{separation sets} are:
\begin{align*}
L_2(c,c')&\coloneqq \{X \in L_2 \: | \: \text{ $c/X,c'/X$ are antipodal chambers in $\mc{A}_X$} \},\\
L_2(r,r')&\coloneqq \{X\in L_2\:|\: X\text{ separates }r\text{ and }r'\}.
\end{align*}
Define an undirected graph $G(c,c')$, whose vertex set is $\mathcal{R}(c,c')$, with an edge between two galleries $r$ and $r'$ if $|L_2(r,r')|=1$. For $r,r'\in\mathcal{R}(c,c')$, write $d(r,r')$ for the distance between $c$ and $c'$ in $G(c,c')$. From the definition it is clear that $d(r,r')\geq|L_2(r,r')|$. 
\begin{defin}\label{def:L2-accessible}
A gallery $r\in\mathcal{R}(c,c')$ is $L_2$-\textit{accessible} if $d(r,r')=|L_2(r,r')|$ for all $r'\in\mathcal{R}(c,c')$.
\end{defin}

Since $\mathcal{A}$ is central, the linear map $x\mapsto -x$ preserves our hyperplane arrangement. Thus, for each chamber $c\in\mathcal{C}$, there is an \textit{opposite chamber} $-c\in\mathcal{C}$. And for each gallery $r\in\mathcal{R}(c,-c)$, there is an \textit{opposite gallery} $-r$ which visits the opposite chambers in the reverse order. By considering two-dimensional central and essential quotient arrangements $\mathcal{A}/X$, we see that $L_2(r,-r)=L_2$ for $r\in\mathcal{R}(c,-c)$, and that $L_2(r,r')\Delta L_2(r,r'')=L_2(r',r'')$ where $\Delta$ denotes the symmetric difference (see also Section 3 of \cite{reiner-roichman}) for $r,r',r''\in\mathcal{R}(c,c')$.

The \emph{diameter} $\diam(G(c,c'))$ of the graph $G(c,c')$ is the maximum distance between two vertices. The following lemma is very helpful.

\begin{lemma}[Proposition 3.12 of \cite{reiner-roichman}]
\label{lem:accessible-implies-equality}
If there exists an $L_2$-accessible gallery $r\in \mathcal{R}(c,-c)$, then $\diam(G(c,-c))=|L_2|$.
\end{lemma}

Following Reiner--Roichman, we will be interested in $\diam(G(c,c'))$ when $\mc{A}$ is the \emph{Coxeter arrangement} for a finite Coxeter group (see Section~\ref{sub:background-rootsystem}).

\subsection{Root systems and Weyl groups}\label{sub:background-rootsystem}

Let $\Phi\subset E$ be a finite crystallographic root system of rank $n$, where $E$ is an ambient Euclidean space of dimension $n$ (see \cite{humphreys} for basic definitions). Choose $\Phi^+\subset\Phi$ to be a set of positive roots whose corresponding simple roots are $\Delta=\{\alpha_1,\ldots,\alpha_n\}\subset\Phi^+$.  For each $\alpha\in\Phi^+$, let $s_{\alpha}$ be the reflection across the hyperplane normal to $\alpha$, and for simplicity, write $s_i$ for the simple reflections $s_{\alpha_i}$. Let $W=W(\Phi)\subset\mathrm{GL}(E)$ be the Weyl group associated to $\Phi$, the group generated by the $s_i$ for $i=1,\ldots,n$.

For each $w\in W$, denote its Coxeter length by $\ell(w)$, its set of reduced words by $\Red(w)$, and its \textit{inversion set} by
\[
I_{\Phi}(w)\coloneqq\{\alpha\in\Phi^+\:|\: w\alpha \not\in\Phi^+\}.
\]
We write $w_0=w_0(\Phi)$ for the unique element of $W$ of maximum length.

We adopt the following conventions for root systems of classical types, where $e_i$ denotes the $i$-th standard basis vector in $\R^n$:
\begin{itemize}
\item Type $A_{n-1}$: $\Phi=\{e_i-e_j\:|\:1\leq i\neq j\leq n\}$, $\Phi^+=\{e_j-e_i\:|\: 1\leq i<j\leq n\}$, $\Delta=\{\alpha_1=e_2-e_1,\ldots,\alpha_n=e_{n}-e_{n-1}\}$, $W\simeq S_n$ the symmetric group.
\item Type $B_n$: $\Phi=\{\pm e_i\pm e_j\:|\: i\neq j\}\cup\{\pm e_i\}$, $\Phi^+=\{e_j-e_i\:|\: i<j\}\cup \{e_i+e_j \: | \: 1 \leq i < j \leq n\} \cup \{e_i \: | \: 1 \leq i \leq n\}$, $\Delta=\{\alpha_1=e_1,\alpha_2=e_2-e_1,\ldots,\alpha_n=e_n-e_{n-1}\}$, $W(B_n)$ is the signed symmetric group:
\[
W(B_n)=\{w\text{ permutation of }-n,\ldots,-1,1,\ldots,n\:|\: w(i)=-w(-i),\forall i\}.
\]
\item Type $D_n$: $\Phi=\{\pm e_i\pm e_j\:|\: i\neq j\}$, $\Phi^+=\{e_j-e_i\:|\: i<j\} \cup \{e_j+e_i \: | \: i<j \}$, $\Delta=\{\alpha_1=e_1+e_2,\alpha_2=e_2-e_1,\alpha_3=e_3-e_2,\ldots,\alpha_n=e_n-e_{n-1}\}$, $W(D_n)$ is an index-two subgroup of $W(B_n)$:
\[
W(D_n)=\{w\in W(B_n)\:|\: w(i)<0\text{ for an even number of }i>0\}.
\]
\end{itemize}
The Dynkin diagrams corresponding to the above conventions are shown in Figure~\ref{fig:Dynkin}.
\begin{figure}[ht]
\centering
\begin{tikzpicture}[scale=0.6]
\node at (0,0) {$\bullet$};
\node[below] at (0,0) {$1$};
\node at (1,0) {$\bullet$};
\node[below] at (1,0) {$2$};
\node at (2,0) {$\bullet$};
\node[below] at (2,0) {$3$};
\node at (3.5,0) {$\bullet$};
\node at (4.5,0) {$\bullet$};
\node at (5.5,0) {$\bullet$};
\node[below] at (5.5,0) {$n{-}1$};
\draw(0,0)--(2,0);
\draw(3.5,0)--(5.5,0);
\node at (2.75,0) {$\cdots$};
\end{tikzpicture}
\quad
\begin{tikzpicture}[scale=0.6]
\node at (0,0) {$\bullet$};
\node[below] at (0,0) {$1$};
\node at (1,0) {$\bullet$};
\node[below] at (1,0) {$2$};
\node at (2,0) {$\bullet$};
\node[below] at (2,0) {$3$};
\node at (3.5,0) {$\bullet$};
\node at (4.5,0) {$\bullet$};
\node[below] at (4.5,0) {$n{-}1$};
\node at (5.5,0) {$\bullet$};
\node[below] at (5.5,0) {$n$};
\draw(0,0.05)--(1,0.05);
\draw(0,-0.05)--(1,-0.05);
\draw(1,0)--(2,0);
\draw(3.5,0)--(5.5,0);
\node at (2.75,0) {$\cdots$};
\end{tikzpicture}
\quad
\begin{tikzpicture}[scale=0.6]
\node at (0,0) {$\bullet$};
\node[below] at (0,0) {$3$};
\node at (1,0) {$\bullet$};
\node[below] at (1,0) {$4$};
\node at (1.75,0) {$\cdots$};
\node at (2.5,0) {$\bullet$};
\node at (3.5,0) {$\bullet$};
\node[below] at (3.5,0) {$n$};
\draw(0,0)--(1,0);
\draw(2.5,0)--(3.5,0);
\node at (-0.9,0.4) {$\bullet$};
\node at (-0.9,-0.4) {$\bullet$};
\node[left] at (-0.9,0.4) {$1$};
\node[left] at (-0.9,-0.4) {$2$};
\draw(-0.9,0.4)--(0,0)--(-0.9,-0.4);
\end{tikzpicture}
\caption{Dynkin diagrams of types $A_{n-1}$, $B_n$, and $D_n$.}
\label{fig:Dynkin}
\end{figure}
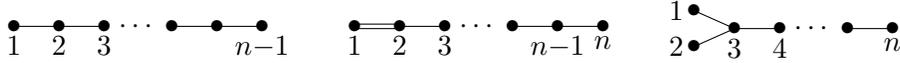

For simplicity of notation, when working in types $B$ and $D$ we will use $\bar i$ to represent $-i$, and write $e_{\bar i}$ for $-e_i$. For a signed permutation $w\in W(B_n)$, its one-line notation is written as $w(1)w(2)\cdots w(n)$. For example, $w=32\bar 14\in W(B_4)$ means that $w(1)=3$, $w(2)=2$, $w(3)=-1$ so that $w(-3)=1$ and $w(4)=4$. 

\begin{defin}
For a reduced word $r=s_{i_1}\cdots s_{i_{\ell}}\in\Red(w)$, its \textit{root ordering} is an ordering $\ro(r)=\beta_1,\ldots,\beta_{\ell}$ of $I_{\Phi}(w)$ where $\beta_j=s_{i_1}\cdots s_{i_{j-1}}\alpha_j\in\Phi^+$. The \textit{reverse root ordering} is $\rro(r)=\beta_{\ell},\ldots,\beta_1$, which will be used more frequently in this paper.
\end{defin}
Here is an alternative way of thinking about the root ordering given a reduced word $r=s_{i_1}\cdots s_{i_{\ell}}$. Let $w^{(j)}=s_{i_1}\cdots s_{i_{j}}$ where $w^{(0)}=\id$ and $w^{(\ell)}=w$. Then $\beta_j$ is the unique positive root $\beta\in\Phi^+$ such that $s_{\beta}w^{(j)}=w^{(j-1)}$. In the classical types, the reduced word $r$, read in reverse, records how $w$ is sorted step by step to $\id$ by swapping adjacent indices, while the reverse root ordering records the values being swapped.

\begin{ex}
\label{ex:type-A-reduced-word}
Consider $w=3412\in W(A_3)$ and pick a reduced word $r=s_2s_1s_3s_2\in\Red(w)$. We compute its root ordering to be $\ro(r)=e_3-e_2,e_3-e_1,e_4-e_2,e_4-e_1$. Viewing $r$ as describing how $w$ is reduced to $\id$ (as $ws_2s_3s_1s_2=\id$), we naturally obtain the reverse root ordering:
\[
\begin{tikzcd}
w=3412 \arrow[r, "e_4-e_1"]& 3142 \arrow[r, "e_4-e_2"] & 3124 \arrow[r, "e_3-e_1"] & 1324 \arrow[r, "e_3-e_2"] & 1234.
\end{tikzcd}
\]
\end{ex}
\begin{ex}
Consider $w=4\bar13\bar2\in W(D_4)$ with a reduced word $r=s_1s_3s_4s_3s_2$. Similarly, its reverse root ordering can be easily obtained by considering how $w$ is transformed into $\id$ by $r$:
\[
\begin{tikzcd}
w=4\bar13\bar2 \arrow[r, "e_4+e_1"]& \bar143\bar2 \arrow[r, "e_4-e_3"] & \bar134\bar2 \arrow[r, "e_4+e_2"] & \bar13\bar24 \arrow[r, "e_3+e_2"] & \bar1\bar234 \arrow[r, "e_2+e_1"] & 1234.
\end{tikzcd}
\]
For example, in the first step in the above diagram, which corresponds to the last simple reflection $s_2$ (where $\alpha_2=e_2-e_1$) in $r$, the value $4$ is swapped with $\bar1$ and thus the corresponding positive root is viewed as $e_4-e_{\bar1}=e_4+e_1$.
\end{ex}

Given a root system $\Phi$ of rank $n$, its \textit{Coxeter arrangement} is: \[\mathcal{A}_{\Phi}=\{H_{\alpha} \text{ for }\alpha\in\Phi^+\},\]
where $H_{\alpha}\coloneqq\{x\in\R^n\:|\: \langle x,\alpha\rangle=0\}$ is the hyperplane normal to $\alpha$.

The Weyl group acts simply transitively on the chambers of $\mc{A}_{\Phi}$, giving a labelling of the chambers by elements of $W$. For $w\in W$, we define $G(w)$ to be $G(\id,w)$ (or, isomorphically, $G(c,wc)$ for any $c \in W$) as in Section~\ref{sub:background-hyperplane}. In a straightforward correspondence, the inversion set $I_{\Phi}(w)$ becomes the set of hyperplanes that separate the chambers $\id$ and $w$. Moreover, reduced words $\Red(w)$ correspond to galleries $\mathcal{R}(\id,w)$ and the root ordering $\ro(r)$ corresponds to the sequence of hyperplanes crossed by the gallery $r$ from the identity chamber $\id$ to $w$. Under this identification, $G(w)$ can alternatively be described as the graph with vertex set $\Red(w)$ and edges connecting reduced words which differ by a single application of a relation
\[
\underbrace{s_is_j \cdots}_{m_{ij}} = \underbrace{s_js_i \cdots}_{m_{ij}},
\]
called a \emph{Coxeter move}. When $m_{ij}=2$, such a move is called a \emph{commutation move}.

Again identifying chambers of $\mc{A}_{\Phi}$ with elements of $W$, we write $L_2(w)$ for the set $L_2(\id,w)$ from Section~\ref{sub:background-hyperplane}. A \emph{root subsystem} $\Psi$ of $\Phi$ is a set of roots $\Phi \cap E'$ where $E'$ is a vector subspace of $E$; this subsystem is a root system in its own right and comes equipped with a natural choice $\Psi^+=\Phi^+ \cap E'$ of positive roots. The separation sets can be understood in terms of root subsystems as follows:
\begin{lemma}\label{lem:L2-reduced-words} \text{}
\begin{itemize}
    \item[(a)] For $w \in W$, $|L_2(w)|$ is the number of rank-two root subsystems $\Psi$ of $\Phi$ with $\Psi^+ \subset I_{\Phi}(w)$.
    \item[(b)] For $r_1,r_2\in\Red(w)$, $|L_2(r_1,r_2)|$ is the number of rank-two root subsystems $\Psi$ of $\Phi$ with $\Psi^+ \subset I_{\Phi}(w)$ such that the roots in $\Psi^+$ appear in different orders in $\ro(r_1)$ and $\ro(r_2)$.
\end{itemize}
\end{lemma}
\begin{proof}
Part (a) appears in \cite{reiner-roichman}. For part (b), viewing $r_1$ and $r_2$ as galleries between chambers $\id$ and $w$, we recall from Section~\ref{sub:background-hyperplane} that a codimension-two subspace $X$ lies in $L_2(r_1,r_2)$ if and only if $r_1/X\neq r_2/X$. A codimension-two subspace $X$ corresponds to a rank-two root subsystem $\Psi\subset\Phi$ since we can write $X=\bigcap_{\alpha \in \Psi} H_{\alpha}$. So $r_1/X\neq r_2/X$ is equivalent to the orders in which $r_1$ and $r_2$ cross the hyperplanes corresponding to $\Phi^+$ differing, which is the same as the roots in $\Psi^+$ appearing in different orders in $\ro(r_1)$ and $\ro(r_2)$.
\end{proof}

\subsection{Permutation patterns}\label{sub:patterns}
Given a permutation $w=w_1 \ldots w_n \in S_n$ and another $p=p_1 \ldots p_k \in S_k$, we say $w$ \emph{contains an occurrence of $p$} in the positions $1 \leq i_1 < \cdots < i_k \leq n$ if $w_{i_1},\ldots,w_{i_k}$ are in the same relative order as $p_1,\ldots,p_k$. We write $N_p(w)$ for the number of occurrences of $p$ in $w$ and say that $w$ \emph{avoids} $p$ if $N_p(w)=0$.

\begin{prop}
\label{prop:inverse-patterns}
Let $w \in S_n$ and $p \in S_k$, then 
\[
N_p(w)=N_{p^{-1}}(w^{-1}).
\]
\end{prop}
\begin{proof}
An easy check shows that $(i_1,\ldots,i_k)$ is an occurrence of $p$ in $w$ if and only if $\sort(w_{i_1},\ldots,w_{i_k})$ is an occurrence of $p^{-1}$ in $w^{-1}$, where $\sort$ is the rearrangement of a tuple into increasing order.
\end{proof}

Simion \cite{simion} introduced a notion of patterns in elements of $W(B_n)$ and $W(D_n)$ generalizing permutation patterns. This will be discussed briefly in Section~\ref{sub:type-B} but the details will not be needed.

\section{The symmetric group case}
\label{sec:symmetric-group}

Our strategy for proving lower bounds on $\diam(G(w))$ is as follows. First, we choose distinguished reduced words $r_1,r_2 \in \Red(w)$ which are in some sense opposite one another. From the definitions this gives a bound
\begin{equation}
\label{eq:diam-bound-from-r1r2}
\diam(G(w)) \geq d(r_1,r_2) \geq \left| L_2(r_1,r_2) \right|.
\end{equation}
Next, we analyze the contributions to $L_2(w)$ and to $L_2(r_1,r_2)$ coming from rank-two subsystems according to their appearances in occurrences of permutation patterns in $w$. This gives formulas:
\begin{align}
\label{eq:L2-bound-setup1}
    \left| L_2(r_1,r_2) \right| &= \sum_p a_p N_p(w), \\
\label{eq:L2-bound-setup2}
    \left| L_2(w) \right| &= \sum_p b_p N_p(w),
\end{align}
where the sums are over certain length three and four permutation patterns $p$ and where the coefficients $a_p,b_p$ do not depend on $w$. Comparing (\ref{eq:L2-bound-setup1}) and (\ref{eq:L2-bound-setup2}) term by term and applying the symmetry of the problem under inversion of $w$, we obtain Theorem~\ref{thm:type-A-lower-bound}.

\subsection{Reiner and Roichman's conjecture}

Throughout this section, for any $w \in W=S_n$ we consider two distinguished reduced words $r_1, r_2 \in \Red(w)$. Viewing reduced words as sequences of swaps taking $w$ to $\id$ as in Example~\ref{ex:type-A-reduced-word}, these are defined as follows:
\begin{itemize}
    \item $r_1$ is the reduced word which first swaps the value 1 leftward as many times as possible to move it to the first position, and then does the same to move the value 2 to the second position, and so on with the values $3,4,\ldots, n$.
    \item $r_2$ is the reduced word which first swaps the value $n$ rightward as many times as possible to move it to the last position, and then does the same to move the value $n-1$ to the penultimate position, and so on with the values $n-2,n-3,\ldots,1$.
\end{itemize}

\begin{figure}[h!]
\centering
\begin{tabular}{|c|c|c|c|c|c|}
\hline 
Pattern $p$ & $A_1 \times A_1$ subsystem $\Psi$ & $\rro(r_1)$ & $\rro(r_2)$ & $a_p$ & $b_p$ \\ \hline
$4321$ & $\R (e_2{-}e_1) \oplus \R (e_4{-}e_3)$  & $e_2{-}e_1$, $e_4{-}e_3$ & $e_4{-}e_3$, $e_2{-}e_1$ & 2 & 3 \\
 & $\R (e_3{-}e_1) \oplus \R (e_4{-}e_2)$  & $e_3{-}e_1$, $e_4{-}e_2$ & $e_4{-}e_2$, $e_3{-}e_1$ & & \\
 & $\R (e_4{-}e_1) \oplus \R (e_3{-}e_2)$  & $e_4{-}e_1$, $e_3{-}e_2$ & $e_4{-}e_1$, $e_3{-}e_2$ & & \\
\hline
$4312$ & $\R (e_4{-}e_1) \oplus \R (e_3{-}e_2)$  &  $e_4{-}e_1$, $e_3{-}e_2$ & $e_4{-}e_1$, $e_3{-}e_2$ & 1 & 2 \\
& $\R (e_3{-}e_1) \oplus \R (e_4{-}e_2)$  & $e_3{-}e_1$, $e_4{-}e_2$ & $e_4{-}e_2$, $e_3{-}e_1$ & & \\
\hline
$4231$ & $\R (e_3{-}e_1) \oplus \R (e_4{-}e_2)$  & $e_3{-}e_1$, $e_4{-}e_2$ & $e_4{-}e_2$, $e_3{-}e_1$ & 2 & 2 \\
& $\R (e_2{-}e_1) \oplus \R (e_4{-}e_3)$  & $e_2{-}e_1$, $e_4{-}e_3$ & $e_4{-}e_3$, $e_2{-}e_1$ & & \\
\hline
$4213$ & $\R (e_2{-}e_1) \oplus \R (e_4{-}e_3)$  & $e_2{-}e_1$, $e_4{-}e_3$ & $e_4{-}e_3$, $e_2{-}e_1$ & 1 & 1 \\
\hline
$4132$ & $\R (e_4{-}e_1) \oplus \R (e_3{-}e_2)$  & $e_4{-}e_1$, $e_3{-}e_2$ & $e_4{-}e_1$, $e_3{-}e_2$ & 0 & 1 \\
\hline
$3421$ & $\R (e_3{-}e_1) \oplus \R (e_4{-}e_2)$  & $e_3{-}e_1$, $e_4{-}e_2$ & $e_4{-}e_2$, $e_3{-}e_1$ & 1 & 2 \\
& $\R (e_4{-}e_1) \oplus \R (e_3{-}e_2)$  & $e_4{-}e_1$, $e_3{-}e_2$ & $e_4{-}e_1$, $e_3{-}e_2$ & & \\
\hline
$3412$ & $\R (e_3{-}e_1) \oplus \R (e_4{-}e_2)$  & $e_3{-}e_1$, $e_4{-}e_2$ & $e_4{-}e_2$, $e_3{-}e_1$ & 1 & 2 \\
& $\R (e_4{-}e_1) \oplus \R (e_3{-}e_2)$  & $e_4{-}e_1$, $e_3{-}e_2$ & $e_4{-}e_1$, $e_3{-}e_2$ & & \\
\hline
$3241$ & $\R (e_4{-}e_1) \oplus \R (e_3{-}e_2)$  & $e_4{-}e_1$, $e_3{-}e_2$ & $e_4{-}e_1$, $e_3{-}e_2$ & 0 & 1 \\
\hline
$3142$ & $\R (e_3{-}e_1) \oplus \R (e_4{-}e_2)$  & $e_3{-}e_1$, $e_4{-}e_2$ & $e_4{-}e_2$, $e_3{-}e_1$ & 1 & 1 \\
\hline
$2431$ & $\R (e_2{-}e_1) \oplus \R (e_4{-}e_3)$  & $e_2{-}e_1$, $e_4{-}e_3$ & $e_4{-}e_3$, $e_2{-}e_1$ & 1 & 1 \\
\hline
$2413$ & $\R (e_2{-}e_1) \oplus \R (e_4{-}e_3)$  & $e_2{-}e_1$, $e_4{-}e_3$ & $e_4{-}e_3$, $e_2{-}e_1$ & 1 & 1 \\
\hline
$2143$ & $\R (e_2{-}e_1) \oplus \R (e_4{-}e_3)$  & $e_2{-}e_1$, $e_4{-}e_3$ & $e_4{-}e_3$, $e_2{-}e_1$ & 1 & 1 \\
\hline
\end{tabular}
\caption{Some permutation patterns $p$ and rank-two subsystems $\Psi$ contained in their inversion sets. The induced reverse root orders corresponding to $r_1$ and $r_2$ are shown, allowing for the calculation of the coefficients $a_p$ and $b_p$. See the proof of Proposition~\ref{prop:table-gives-coefficients}.}
\label{fig:type-A-patterns}
\end{figure}

\begin{prop}
\label{prop:table-gives-coefficients}
With $r_1,r_2$ as above, the coefficients $a_p, b_p$ in (\ref{eq:L2-bound-setup1}) and (\ref{eq:L2-bound-setup2}) are the following:
\begin{itemize}
    \item[(i)] $a_{321}=b_{321}=1$,
    \item[(ii)] $a_p,b_p$ are given by the values in Figure~\ref{fig:type-A-patterns} for the twelve length-four patterns $p$ shown there, and
    \item[(iii)] $a_p=b_p=0$ for all other patterns $p$.
\end{itemize}
\end{prop}
\begin{proof}
Lemma~\ref{lem:L2-reduced-words} expresses $|L_2(w)|$ as the number of rank-two subsystems $\Psi \subset \Phi$ with $\Psi^+ \subset I_{\Phi}(w)$ and $|L_2(r_1,r_2)|$ as the number of these which appear in different orders in $\rro(r_1)$ and $\rro(r_2)$. Equations (\ref{eq:L2-bound-setup2}) and (\ref{eq:L2-bound-setup1}) group these subsystems according to the length three and four patterns in which they occur, with the coefficients $a_p$ and $b_p$ recording the number of such subsystems contained in an occurrence of a given pattern.

Since $\Phi$ is of type $A_{n-1}$ the only possible rank-two subsystems are of types $A_2$ and $A_1 \times A_1$. Occurrences of the pattern $321$ in $w$ as $w_{i_1}\ldots w_{i_2}\ldots w_{i_3}$ with $i_1<i_2<i_3$ and $w_{i_1}>w_{i_2}>w_{i_3}$ correspond exactly to the type $A_2$ subsystems; the corresponding roots are $e_{w_{i_1}}-e_{w_{i_2}}, e_{w_{i_1}}-e_{w_{i_3}},$ and $e_{w_{i_2}}-e_{w_{i_3}}$. By construction, such triples of roots will always appear in the opposite order in $\rro(r_1)$ as in $\rro(r_2)$, so $a_{321}=b_{321}=1$, proving (i).

Subsystems $\Psi$ of type $A_1 \times A_1$ correspond to collections of distinct indices $i_1<i_2, j_1<j_2$ with $w_{i_1}>w_{i_2}$ and $w_{j_1}>w_{j_2}$. There is an occurrence of one of the twelve patterns $p$ from Figure~\ref{fig:type-A-patterns} in $w$ in positions $\{i_1,i_2,j_1,j_2\}$. For each possible pattern $p$, the table lists the subsystems that correspond to an occurrence of it (the indices $1,2,3,4$ given in the table should be replaced with the positions $k_1<k_2<k_3<k_4$ in which $p$ occurs in $w$). By Lemma~\ref{lem:L2-reduced-words}, the coefficient $b_p$ is the total number of these subsystems for $p$, while $a_p$ is the number for which the roots of $\Psi^+$ occur in different orders in $\rro(r_1)$ and $\rro(r_2)$; these orders are given in the table, allowing $a_p$ to be calculated and proving (ii). 

Since we are only interested in rank-two subsystems, there are no other patterns to consider, resolving (iii).
\end{proof}

We are now ready to complete the proof of Theorem~\ref{thm:type-A-lower-bound}.

\begin{proof}[Proof of Theorem~\ref{thm:type-A-lower-bound}]
We would like to prove this bound by comparing the coefficients $a_p,b_p$ appearing in (\ref{eq:L2-bound-setup1}) and (\ref{eq:L2-bound-setup2}). Figure~\ref{fig:type-A-patterns} shows that $a_p \geq \frac{1}{2} b_p$ except in the cases $p=4132,3241$. We address these exceptions by exploiting the symmetry of the problem under inversion:
\begin{align*}
    \diam(G(w))&=\diam(G(w^{-1}))  \\
    & \geq \frac{1}{2} \sum_p a_p N_p(w) + \frac{1}{2} \sum_p a_p N_p(w^{-1}) \\
    &= \frac{1}{2} \sum_p (a_p+a_{p^{-1}}) N_p(w) \\
    & \geq \frac{1}{2} \sum_p b_p N_p(w) \\
    &= \frac{1}{2} \left| L_2(w) \right|.
\end{align*}
The first line follows from the definitions because reduced words of $w^{-1}$ are just the reverses of reduced words of $w$; the second follows from this after applying (\ref{eq:diam-bound-from-r1r2}) and (\ref{eq:L2-bound-setup1}); the third follows from Proposition~\ref{prop:inverse-patterns}; the fourth follows by noting that Proposition~\ref{prop:table-gives-coefficients} yields $a_p+a_{p^{-1}} \geq b_p$ for all $p$; and the last follows from (\ref{eq:L2-bound-setup2}).
\end{proof}
\subsection{Dahlberg and Kim's conjecture}

\begin{proof}[Proof of Theorem~\ref{thm:type-A-equality-cases}]
The inequality
\begin{equation}
    \frac{1}{2} \sum_p (a_p+a_{p^{-1}}) N_p(w) \geq \frac{1}{2} \sum_p b_p N_p(w)
\end{equation}
from the proof of Theorem~\ref{thm:type-A-lower-bound} is strict unless $w$ avoids all patterns $p$ from
\[
P=\{321, 3142, 2413, 2143\},
\]
since $a_p+a_{p^{-1}} > b_p$ in these cases. Thus, since Dahlberg and Kim already proved that permutations of the form
\begin{equation}
\label{eq:equality-form}
\id_i \oplus (\id_j \ominus \id_k) \oplus \id_{\ell}
\end{equation}
achieve equality in Theorem~\ref{thm:type-A-lower-bound}, it suffices to prove that any $w$ avoiding all the patterns from $P$ is of the form (\ref{eq:equality-form}).

Suppose that $w$ avoids the patterns from $P$. A permutation avoiding $3142$ and $2413$ is called \emph{separable}, and it is known \cite{separable1, wei} that separable permutations may be built up as iterated direct sums and skew sums starting from the permutation $\id_1 \in S_1$. We consider two cases according to weather $w$ is a direct sum or skew sum of smaller separable permutations: 

\emph{Case 1:} If $w=w^{(1)} \oplus \cdots \oplus w^{(m)}$ and this direct sum cannot be refined further, then at most one of the summands decomposes further as a proper skew sum, for otherwise $w$ would contain the pattern $2143$. Thus all but one of the summands are $\id_1$. The remaining summand, say $w^{(c)}$, is itself separable and decomposes as 
\[
w^{(c)}=u^{(1)} \ominus \cdots \ominus u^{(m')}.
\]
We must have $m' \leq 2$, otherwise $w$ would contain the pattern $321$. For the same reason, if $m'=2$ we must have $u^{(1)}=\id_j$ and $u^{(2)}=\id_k$ for some $j,k \geq 1$. Thus in this case we see that $w$ is of the form (\ref{eq:equality-form}) with $i=c-1$ and $\ell=m-c$.

\emph{Case 2:} If instead $w=w^{(1)} \ominus \cdots \ominus w^{(m)}$ and this skew sum cannot be refined further, we must have $m \leq 2$ and $w^{(1)}=\id_j$ and $w^{(2)}=\id_k$ for the same reason as for $w^{(c)}$ in the previous case. Thus $w$ is of the form (\ref{eq:equality-form}) for $i=\ell=0$.
\end{proof}

\subsection{A conjectured upper bound}
\label{sub:3412-conjecture}
Conjecture~\ref{conj:3412} below, based on computational evidence, strengthens both Reiner and Roichman's conjectured upper bound (Conjecture 5.9 of \cite{reiner-roichman}) and Dahlberg and Kim's conjecture regarding the pattern 3412 (Conjecture 7.2 of \cite{dahlberg-kim}).

\begin{conj}\label{conj:3412}
For any permutation $w\in S_n$, we have
\[
\diam(G(w))\leq |L_2(w)|-N_{3412}(w).
\]
\end{conj}

The intuition for Conjecture~\ref{conj:3412} is that every occurrence of the pattern 3412 in $w$ contains a rank-two root subsystem $\{e_{i_3}{-}e_{i_2},e_{i_4}{-}e_{i_1}\}$ that contributes to $|L_2(w)|$, but does not correspond to any edge in $G(w)$. 

\section{Comments on lower bounds for other classical types}
\label{sec:other-types}
\subsection{Type $B_n$}
\label{sub:type-B}
For $w \in W(B_n)$, Reiner and Roichman conjecture that
\begin{equation}
\label{eq:type-B-lower-bound}
\frac{1}{3}\left| L_2(w) \right| \leq \diam(G(w)).
\end{equation}
This is very similar to their conjectured lower bound (\ref{eq:main-thm-statement}) in type $A$ which was proven in Section~\ref{sec:symmetric-group}.

The strategy given in Section~\ref{sec:symmetric-group} makes sense in type $B$ as well, but in this case it is insufficient to prove the full strength of the conjecture. Following the same argument, with Simion's type $B$ patterns \cite{simion} replacing permutation patterns, instead gives the weaker bound
\[
\frac{1}{4} \left| L_2(w) \right| - \frac{1}{2}N_{132\bar{4}}(w)-\frac{1}{2}N_{\bar{1}32\bar{4}}(w) \leq \diam(G(w))
\]
in type $B$. The coefficient-by-coefficient approach used in Section~\ref{sec:symmetric-group} for comparing (\ref{eq:L2-bound-setup1}) and (\ref{eq:L2-bound-setup2}) no longer suffices, and some other understanding of these quantities is needed.

\subsection{Type $D_n$}
\label{sub:lower-bounds-D}
The following example shows that there is no uniform lower bound
\[
\varepsilon |L_2(x)| \leq \diam(G(x))
\]
for $\varepsilon>0$ in type $D$, unlike in type $A$ (Theorem~\ref{thm:type-A-lower-bound}), and unlike the conjecture (\ref{eq:type-B-lower-bound}) in type $B$.

\begin{ex}
\label{ex:no-lower-bound-type-D}
Label the Dynkin diagram of type $D_n$ as in Figure~\ref{fig:Dynkin}, and let
\[
w=s_n s_{n-1} \cdots s_4 s_3 s_2 s_1 s_3 s_4 \cdots s_{n-1}s_n.
\]
This is the longest element $w_0^J$ for the parabolic quotient of $W$ corresponding to $J=\{s_1,\ldots, s_{n-1}\}$. As the only move which can be applied to the given reduced word is the commutation of $s_1$ and $s_2$, and as no more moves are applicable after that, it is clear that $\diam(G(w))=1$. 

The inversion set for $w$ is:
\[
I_{\Phi}(w) = \{e_n-e_i \: | \: 1 \leq i \leq n-1\} \cup \{e_n+e_i \: | \: 1 \leq i \leq n-1\}.
\]
The only rank-two subsystems whose positive roots are contained in $I_{\Phi}(w)$ are thus the $n-1$ subsystems of type $A_1 \times A_1$ with positive roots $\{e_n-e_i,e_n+e_i\}$ for $i=1,\ldots,n-1$. Thus $\left| L_2(w) \right|=n-1$. In particular, taking $n$ large we see that there is no uniform lower bound for $\diam(x)/|L_2(x)|$ in type $D$.
\end{ex}

\section{Upper bounds in type $D_n$}
\label{sec:upper-bounds-D}
Studying $\diam(G(w_0))$ is of particular interest in light of the following open problem: 
\begin{quest}[Reiner and Roichman \cite{reiner-roichman}]
\label{quest:w0-question}
Do we have
\[
\diam(G(w_0))=|L_2|
\]
in every finite Coxeter group? 
\end{quest}
This question was answered affirmatively for types $A_n$ and $B_n$ and the dihedral groups by Reiner and Roichman \cite{reiner-roichman} but remains open for the remaining infinite family $D_n$ and for several exceptional types.

In all cases in which Question~\ref{quest:w0-question} has been resolved, this has been done by exhibiting an accessible reduced word and applying Lemma~\ref{lem:accessible-implies-equality}. The following observation, confirmed via computer checks, shows that this method will be insufficient in general.

\begin{prop}\label{prop:no-accessible-D5}
There does not exist any $L_2$-accessible reduced word for $w_0(D_5)$.
\end{prop}


In this section, we show, however, that there exists a reduced word $r\in\Red(w_0(D_n))$ that is close to being $L_2$-accessible. The construction of $r$ is analogous to the reduced words considered in Section~\ref{sec:symmetric-group}. 

We recall that in type $D_n$, $w_0(1)=(-1)^{n+1}$ and $w_0(i)=\bar i$ for $2\leq i\leq n$. Also recall that $\alpha_1=e_2+e_1$, $\alpha_i=e_i-e_{i-1}$ for $2\leq i\leq n$ so that $s_1$ and $s_2$ commute. Let \[
r=(s_1s_2)(s_3s_2s_1s_3)(s_4s_3s_1s_2s_3s_4)\cdots (s_ns_{n-1}\cdots s_{\frac{3-(-1)^n}{2}}s_{\frac{3+(-1)^n}{2}}s_3\cdots s_n),
\]
where the order of the product $s_1s_2$ in each factor above alternates. Viewing $r$ as decreasing $w$ to $\id$, it first sends $n$ and $\bar n$ to their positions, by swapping $n$ with $n-1,n-2,\ldots,1,\bar1,\ldots,\overline{n-1}$ in this order, and then sends $n-1$ and $\overline{n-1}$ to their positions and so on. An easy check shows that the reverse root ordering is:
\[
\rro(r)=e_n{-}e_{n-1},\ldots,e_n{-}e_1,e_n{+}e_1,\ldots,e_n{+}e_1,\ldots,e_n{+}e_{n-1},e_{n-1}{-}e_{n-2},\ldots,e_2{+}e_1.
\]

\begin{theorem}\label{thm:Dn-close-accessible}
Let $r$ be as above. For any reduced word $r'\in\Red(w_0(D_n))$,
\[
d(r',r)< |L_2(r',r)| + \frac{2}{3}n^3.
\]
\end{theorem}
Theorem~\ref{thm:Dn-close-accessible} says that $r\in\Red(w_0(D_n))$ is almost $L_2$-accessible, since $|L_2|=\frac{1}{6}n(n-1)(3n^2-11n+13)=\frac{1}{2}n^4+O(n^3)$.
\begin{proof}[Proof of Theorem~\ref{thm:Dn-close-accessible}]
For $j=2,\ldots,n$, let
\[
A_j=\{e_j\pm e_i\:|\: 1\leq i<j\}\subset\Phi^+(D_n).
\]
Thus, we have a partition of positive roots $\Phi^+(D_n)=\bigsqcup_{j=2}^n A_j.$

Our strategy is the following. Given $r'$, we will gradually move it towards $r$ and keep track of the change in $d(-,r)$ and $|L_2(-,r)|$. For $r'\in\Red(w_0)$, define \[f(r')\coloneqq d(r',r)-|L_2(r',r)|.\]

Let $r'=s_{i_{\ell}}\cdots s_{i_1}\in\Red(w_0)$ be an explicit reduced word and $\rro(r')=\gamma_1,\gamma_2,\ldots,\gamma_{\ell}$ be the reverse root ordering of $r'$, where $\ell=n^2-n$ is the number of positive roots in type $D_n$. Let $u^{(j)}=ws_{i_1}\cdots s_{i_{j}}$ so that $u^{(0)}=w$ and $u^{(\ell)}=\id$. 

Let $t=t(r')$ be the largest integer such that $\gamma_t\notin A_n$ and $\gamma_{t+1}\in A_n$. If such $t$ does not exist, then $\{\gamma_1,\ldots,\gamma_{2n-2}\}=A_n$, which is a degenerate case that will be discussed later. Let $k\geq1$ be the integer such that $\gamma_{t+1},\ldots,\gamma_{t+k}\in A_n$ and $\gamma_{t+k+1}\notin A_n$. By the definition of $t$, we have $\gamma_{t+k+1},\ldots,\gamma_{\ell}\notin A_n$. In other words, we focus on the last block of consecutive roots $\gamma_{t+1},\ldots,\gamma_{t+k}$ in $\rro(r')$ that belong to $A_n$. 

Let $j=(u^{(t)})^{-1}(n)$ be the index of $n$ in the signed permutation $u^{(t)}$. From $u^{(t)}$ to $u^{(t+k)}=u^{(t)}s_{i_{t+1}}\cdots s_{i_{t+k}}$, the value $n$ is moved consecutively from index $j$ to index $n$. Note that since $\gamma_{t}\notin A_n$, the simple transposition $s_{i_t}$ does not involve index $j$ (or $\bar j$). 
\begin{claim}\label{claim:D4-one-step}
We can locally change $s_{i_t}s_{i_{t+1}}\cdots s_{i_{t+k}}$ to obtain another reduced word $r''\in\Red(w_0)$ such that
\begin{equation}\label{eq:rror''}
    \rro(r'')=\gamma_1,\ldots,\gamma_{t-1},\underline{\gamma_{t+1}',\ldots,\gamma_{t+k}',\gamma_{t}},\gamma_{t+k+1},\ldots,\gamma_{\ell}
\end{equation}
where $\gamma_{t+1}',\ldots,\gamma_{t+k}'$ is a permutation of $\gamma_{t+1},\ldots,\gamma_{t+k}$, and such that $f(r')\leq f(r'')+2$. 
\end{claim}
\begin{proof}[Proof of Claim~\ref{claim:D4-one-step}]
Consider the following cases.

\noindent\textbf{Case 1:} $j\geq1$. In this case, $s_{i_{t+1}},\ldots,s_{i_{t+k}}=s_{j+1},\ldots,s_n$. As $\gamma_t\notin A_n$, we also know that ${i_t}\neq j, j+1$.

Subcase 1.1: $i_t\leq j-1$. Then $s_{i_t}$ commutes with $s_{i_{t+1}},\ldots,s_{i_{t+k}}$. Change $s_{i_t},\ldots,s_{i_{t+k}}$ in $r'$ to $s_{i_{t+1}},\ldots,s_{i_{t+k}},s_{i_t}$ to obtain $r''$ and we easily see that $\rro(r'')$ has the form as above in Equation~\eqref{eq:rror''}. Moreover, $L_2(r'',r)$ is exactly $L_2(r',r)$ with the $k$ root subsystems of type $A_1\times A_1$ generated by $\gamma_t$ and $\gamma_{t+a}$ for $a=1,\ldots,k$ removed, so $|L_2(r',r)|-|L_2(r'',r)|=k$. We need $k$ commutation moves to obtain $r''$ from $r'$ so $d(r',r)-d(r'',r)\leq d(r',r'')\leq k$. This gives $f(r')\leq f(r'')$. 

Subcase 1.2: $i_t\geq j+2$. We can use commutation moves, then a Coxeter move $s_{i_t}s_{i_t-1}s_{i_t}=s_{i_t-1}s_{i_t}s_{i_t-1}$, then commutation moves to locally change $s_{i_t},s_j,\ldots,s_n$ to $s_j,\ldots,s_n,s_{i_t-1}$. We check that in this case $\rro(r'')$ satisfies Equation~\eqref{eq:rror''}. Here, $L_2(r'',r)$ is exactly $L_2(r',r)$ taken away $k-2$ root subsystems of type $A_1\times A_1$ and one root subsytem of type $A_2$, so $|L_2(r',r)|-|L_2(r'',r)|=k-1$. As $r''$ can be obtained from $r'$ via $k-2$ commutation moves and one Coxeter move, we conclude that $d(r',r)-d(r'',r)\leq d(r',r'')\leq k-1$ so $f(r')\leq f(r'')$. Table~\ref{tab:r'tor''-case1.2} shows an example of the change from $r'$ to $r''$.
\begin{table}[h!]
\renewcommand*{\arraystretch}{1.2}
\centering
\begin{tabular}{c|c|c|c|c|c|c|}
 & $u^{(t-1)}$ & $u^{(t)}$ & $\cdots$ & $\cdots$ & $\cdots$ & $u^{(t+k)}$  \\\hline
$r'$ & $\bar3612\bar45$ & $\bar361\bar425$ & $\bar316\bar425$ & $\bar31\bar4625$ & $\bar31\bar4265$ & $\bar31\bar4256$ \\
$\rro(r')$ & & $e_4+e_2$ & $e_6-e_1$ & $e_6+e_4$ & $e_6-e_2$ & $e_6-e_5$ \\\hline
$r''$ & $\bar3612\bar45$ & $\bar3162\bar45$ & $\bar3126\bar45$ & $\bar312\bar465$ & $\bar312\bar456$ & $\bar31\bar4256$ \\
$\rro(r'')$ & & $e_6-e_1$ & $e_6-e_2$ & $e_6+e_4$ & $e_6-e_5$ & $e_4+e_2$ \\\hline
\end{tabular}
\caption{An example of how $r''$ is obtained from $r'$ in Subcase 1.2, where $n=6$, $j=2$, $i_t=5$.}
\label{tab:r'tor''-case1.2}
\end{table}

The majority of future cases require the same analysis of $f(r'')$ as in Subcase 1.1 and Subcase 1.2. We will omit details for those situations.

\noindent\textbf{Case 2:} $j=-1$. In this case, $s_{i_{t+1}},\ldots,s_{i_{t+k}}=s_1,s_3,s_4,\ldots,s_n$ and $i_t\neq 1,2$. As in Subcase 1.2, we can use commutation moves, one Coxeter move of the form $s_{i_t}s_{i_t-1}s_{i_t}=s_{i_t-1}s_{i_t}s_{i_t-1}$ followed by commutations moves to obtain $\rro(r'')$ as in Equation~\eqref{eq:rror''}. The same analysis as in Subcase 1.2 shows that $|L_2(r',r)|-|L_2(r'',r)|=k-1$ and $d(r',r)-d(r'',r)\leq k-1$ so $f(r')\leq f(r'')$.

\noindent\textbf{Case 3:} $j=-2$. In this case, $s_{i_{t+1}},\ldots,s_{i_{t+k}}=s_1,s_2,s_3,s_4,\ldots,s_n$ or $s_2,s_1,s_3,s_4,\ldots,s_n$ and $i_t\neq 1,2$. We run through the same analysis as in Case 2 to obtain $f(r')\leq f(r'')$.

\noindent\textbf{Case 4:} $j\leq -3$. We know that $i_t\neq -j,-j{+}1$. In this case, $s_{i_{t+1}},\ldots,s_{i_{t+k}}$ is equal to $s_{-j},s_{-j-1},\ldots,$ $s_3,s_1,s_2,s_3,\ldots,s_n$ or to $s_{-j},s_{-j-1},\ldots,s_3,s_2,s_1,s_3,s_4,\ldots,s_n$. 

Subcase 4.1: $i_t\geq3$. The analysis here is exactly the same as in Cases 2 and 3, for which we conclude that $f(r')\leq f(r'')$.

Subcase 4.2: $i_t=1$ or $2$. This is the critical case of the entire proof that makes type $D$ different from type $A$ and $B$. By the symmetry of the Dynkin diagram swapping nodes $1$ and $2$, we can assume without loss of generality that $i_t=1$. We replace the relevant positions of $s_{i_t},s_{i_{t+1}},\ldots,s_{i_{t+k}}$ of $s_{i_t}=s_1,s_{-j},s_{-j-1},\ldots,s_3,(s_1,s_2)/(s_2,s_1),s_3,\ldots,s_n$ by $s_{-j},\ldots,s_3,(s_1,s_2)/(s_2,s_1),s_3,\ldots,s_n,s_2$ to obtain $r''$. We check that $\rro(r'')$ has the form in Equation~\eqref{eq:rror''} since the value $n$ is now moved to its index $n$ from index $-j$ before the values $u^{(t-1)}(1)=:a$ and $u^{(t-1)}(2)=:b$ are swapped. 

As before, $L_2(r',r)$ contains $L_2(r'',r)$ and their difference consists of $2$ root subsystems of type $A_2$: $\{e_n-e_a,e_n-e_b,\pm(e_a-e_b)\}$ and $\{e_n+e_a,e_n+e_b,\pm(e_a-e_b)\}$; and $k-4$ root subsystems of type $A_1\times A_1$ generated by $\pm(e_a-e_b)$ with those $\gamma_{t+k'}$'s not involving coordinates $\pm a,\pm b$. This means $|L_2(r',r)|-|L_2(r'',r)|=k-2$. As for $d(r',r'')$, we have two possibilities:
\begin{align*}
&s_1s_{-j}\cdots s_3s_1s_2s_3\cdots s_n=s_{-j}\cdots s_4s_1s_3s_1s_2s_3\cdots s_n\\
=&s_{-j}\cdots s_4s_3s_1s_3s_2s_3\cdots s_n=s_{-j}\cdots s_4s_3s_1s_2s_3s_2\cdots s_n\\
=&s_{-j}\cdots s_4s_3s_1s_2s_3s_4\cdots s_ns_2
\end{align*}
with $k-4$ commutation moves and 2 Coxeter moves, or

\begin{align*}
&s_1s_{-j}\cdots s_3s_2s_1s_3\cdots s_n=s_1s_{-j}\cdots s_3s_1s_2s_3\cdots s_n\\
=&s_{-j}\cdots s_4s_1s_3s_1s_2s_3\cdots s_n=s_{-j}\cdots s_4s_3s_1s_3s_2s_3\cdots s_n\\
=&s_{-j}\cdots s_4s_3s_1s_2s_3s_2\cdots s_n=s_{-j}\cdots s_4s_3s_1s_2s_3s_4\cdots s_ns_2\\
=&s_{-j}\cdots s_4s_3s_2s_1s_3s_4\cdots s_ns_2
\end{align*}
with $k-2$ commutation moves and 2 Coxeter moves. Both situations give $d(r',r'')\leq k$. Therefore, $f(r')\leq f(r'')+2$.
\end{proof}

Let $r''\in\Red(w_0)$ be as in Claim~\ref{claim:D4-one-step}. Recall that $t(r'')$ is the largest integer such that the $t(r'')^{th}$ root of $\rro(r'')$ is not in $A_n$, but the $t(r'')^{th}+1$ root of $\rro(r'')$ is in $A_n$. We necessarily have $t(r'')<t(r')$. Continue to apply Claim~\ref{claim:D4-one-step} to $r''$, we eventually arrive at some reduced word $r'''\in\Red(w_0)$ such that $t(r''')=0$, i.e. $A_n$ appears as the first $2n-2$ roots in $\rro(r''')$. The number of times that we apply Claim~\ref{claim:D4-one-step} is the number of roots among $\gamma_1,\ldots,\gamma_{t}$ that are not in $A_n$, which is at most $|\Phi^+\setminus A_n|=(n-1)(n-2)$. Therefore, $f(r')\leq f(r''')+2(n-1)(n-2)$. 

Thus, the first $2n-2$ roots in $\rro(r''')$ and $\rro(r)$ must be the same, or differ by a swap of $e_n-e_1$ and $e_n+e_1$ in the $(n-1)^{th}$ and $n^{th}$ position, which corresponds to the commutation move $s_1s_2=s_2s_1$. Such commutation move on $r'''$ decreases $|L_2(r''',r)|$ by 1 so $f(r''')$ does not decrease. As a result, for our purposes, we can assume that $\rro(r''')$ and $\rro(r)$ are equal in the first $2n-2$ roots.

We are now reduced to the case of $D_{n-1}$. By applying the whole process to $n-1,n-2,\ldots$, we arrive at
\[
f(r')\leq f(r)+\sum_{j=1}^n2(j-1)(j-2)=\frac{2}{3}(n-2)(n-1)n.
\]
This means $d(r',r)< |L_2(r',r)|+\frac{2}{3}n^3$ as desired.
\end{proof}

\begin{cor}\label{cor:D4-diameter-upper-bound}
For type $D_n$, $\diam(G(w_0))< |L_2|+\frac{4}{3}n^3$.
\end{cor}
\begin{proof}
We use ideas as in the proof of Proposition 3.12 in \cite{reiner-roichman}. Let $r\in\Red(w_0)$ be the reduced word defined above in the current section. Recall from Section~\ref{sub:background-hyperplane} that for any $x\in\Red(w_0)$ viewed as a gallery in the Coxeter arrangement, there exists an opposite gallery $-x\in\Red(w_0)$ such that $L_2=L_2(x,-x)=L_2(x,z)\sqcup L_2(-x,z)$ for any $z\in\Red(w_0)$ where $\sqcup$ means disjoint union. Thus, for any $x,y\in\Red(w_0)$, by the triangle inequality and Theorem~\ref{thm:Dn-close-accessible}, we have
\begin{align*}
2d(x,y)\leq& d(x,r)+d(r,y)+d(x,-r)+d(-r,y)\\
<& |L_2(x,r)|+|L_2(r,y)|+|L_2(x,-r)| + |L_2(-r,y)|+\frac{8}{3}n^3\\
=& \big(|L_2(x,r)|+|L_2(x,-r)|\big)+\big(|L_2(r,y)|+|L_2(-r,y)|\big)+\frac{8}{3}n^3\\
=&2|L_2|+\frac{8}{3}n^3,
\end{align*}
so $d(x,y)<|L_2|+\frac{4}{3}n^3$ for any $x,y\in\Red(w_0)$. Thus, $\diam(G(w_0))<|L_2|+\frac{4}{3}n^3$.
\end{proof}

\bibliographystyle{plain}
\bibliography{main}

\begin{thebibliography}{10}

\bibitem{dahlberg-kim}
Samantha Dahlberg and Younghwan Kim.
\newblock Diameters of graphs on reduced words of 12 and 21-inflations.
\newblock 2020.
\newblock arXiv:2010.15758 [math.CO].

\bibitem{dehornoy-autord}
Patrick Dehornoy and Marc Autord.
\newblock On the distance between the expressions of a permutation.
\newblock {\em European J. Combin.}, 31(7):1829--1846, 2010.

\bibitem{separable1}
Christian Gaetz and Yibo Gao.
\newblock Separable elements in {W}eyl groups.
\newblock {\em Adv. in Appl. Math.}, 113:101974, 23, 2020.

\bibitem{humphreys}
James~E. Humphreys.
\newblock {\em Reflection groups and {C}oxeter groups}, volume~29 of {\em
  Cambridge Studies in Advanced Mathematics}.
\newblock Cambridge University Press, Cambridge, 1990.

\bibitem{reiner-degree}
Victor Reiner.
\newblock Note on the expected number of {Y}ang-{B}axter moves applicable to
  reduced decompositions.
\newblock {\em European J. Combin.}, 26(6):1019--1021, 2005.

\bibitem{reiner-roichman}
Victor Reiner and Yuval Roichman.
\newblock Diameter of graphs of reduced words and galleries.
\newblock {\em Trans. Amer. Math. Soc.}, 365(5):2779--2802, 2013.

\bibitem{simion}
Rodica Simion.
\newblock Combinatorial statistics on type-{B} analogues of noncrossing
  partitions and restricted permutations.
\newblock {\em Electron. J. Combin.}, 7:Research Paper 9, 27, 2000.

\bibitem{stanley-reduced-words}
Richard~P. Stanley.
\newblock On the number of reduced decompositions of elements of {C}oxeter
  groups.
\newblock {\em European J. Combin.}, 5(4):359--372, 1984.

\bibitem{Stanley-hyperplane}
Richard~P. Stanley.
\newblock An introduction to hyperplane arrangements.
\newblock In {\em Geometric combinatorics}, volume~13 of {\em IAS/Park City
  Math. Ser.}, pages 389--496. Amer. Math. Soc., Providence, RI, 2007.

\bibitem{tits-local}
J.~Tits.
\newblock A local approach to buildings.
\newblock In {\em The geometric vein}, pages 519--547. Springer, New
  York-Berlin, 1981.

\bibitem{tits-words}
Jacques Tits.
\newblock Le probl\`eme des mots dans les groupes de {C}oxeter.
\newblock In {\em Symposia {M}athematica ({INDAM}, {R}ome, 1967/68), {V}ol. 1},
  pages 175--185. Academic Press, London, 1969.

\bibitem{wei}
Fan Wei.
\newblock Product decompositions of the symmetric group induced by separable
  permutations.
\newblock {\em European J. Combin.}, 33(4):572--582, 2012.

\end{thebibliography}
\end{document}